\documentclass[12pt]{amsart}

\usepackage{amssymb}

\usepackage{enumitem}

\usepackage{graphicx}

\makeatletter
\@namedef{subjclassname@2020}{%
  \textup{2020} Mathematics Subject Classification}
\makeatother

\usepackage[T1]{fontenc}

\newtheorem{theorem}{Theorem}[section]
\newtheorem{corollary}[theorem]{Corollary}
\newtheorem{lemma}[theorem]{Lemma}
\newtheorem{proposition}[theorem]{Proposition}
\newtheorem{problem}[theorem]{Problem}

\theoremstyle{definition}

\newtheorem{remark}[theorem]{Remark}
\newtheorem{example}[theorem]{Example}

\numberwithin{equation}{section}

\begin{document}


\baselineskip=17pt


\title[A study on state spaces in classical Banach spaces]{A study on state spaces in classical Banach spaces}

\author[S. Daptari]{Soumitra Daptari}
\address{Department of Mathematics\\Shiv Nadar Institution of Eminence\\ Gautam Buddha Nagar Delhi NCR, Uttar Pradesh-201314, India}
\email{daptarisoumitra@gmail.com}

\author[S. Dwivedi]{Saurabh Dwivedi}
\address{Department of Mathematics\\Shiv Nadar Institution of Eminence\\ Gautam Buddha Nagar Delhi NCR, Uttar Pradesh-201314, India}
\email{sd605@snu.edu.in}

\date{}

\begin{abstract}
Let $X$ be a real or complex Banach space. Let $S(X)$ denote the unit sphere of $X$. For $x\in S(X)$, let $S_{x}=\{x^*\in S(X^*):x^*(x)=1\}$. A lot of Banach space geometry can be determined by the `quantum' of the state space $S_{x}$. In this paper, we mainly study the norm compactness and weak compactness of the state space in the space of Bochner integrable function and $c_{0}$-direct sums of Banach spaces. Suppose $X$ is such that $X^*$ is  separable and let $\mu$ be the Lebesgue measure on $[0,1]$. For $f\in L^1(\mu,X)$, we demonstrate that if $S_{f}$ is norm compact, then $f$ is a smooth point. When $\mu$ is the discrete measure, we show that if $ (x_i) \in S(\ell^{1}(X))$ and $ \|x_{i}\|\neq 0$ for all $i\in{\mathbb{N}}$, then $ S_{(x_i)}$ is weakly compact in $ \ell^\infty(X^*) $ if and only if $ S_{\frac{x_i}{\|x_i\|}} $ is weakly compact in $X^*$ for each $i\in{\mathbb{N}}$ and $\textup{diam}\left(S_{\frac{x_i}{\|x_i\|}}\right) \to 0 $. For discrete $c_{0}$-sums, we show that for $(x_{i})\in c_{0}(X)$, $S_{(x_{i})}$ is weakly  compact if and only if for each $i_{0}\in \mathbb{N}$ such that $\|x_{i_{0}}\|=1$, the state space $S_{x_{i_{0}}}$ is weakly compact.
\end{abstract}


\maketitle
\textbf{To appear in Colloquium Mathematicum.}
\section{Introduction}
Let $X$ be a real or complex Banach space, and let $X^*$ be the dual space of $X$. We denote by $X_1$, the unit ball of $X$. Let $S(X)$ be the unit sphere of $X$. For $x\in S(X)$, let $S_x=\{x^*\in S(X^*):x^*(x)=1\}$ be the state space. By Hahn-Banach theorem, $S_x$ is always non-empty. A subset $F$ of $X_1^*$ is said to be a face of $X_1^*$ if for every $a,b\in X_1^*$ such that for the line segment $[a,b]$, if $[a,b]\cap F\neq \phi$, then $[a,b]\subseteq F$ (here $[a,b]=\{ta+(1-t)b:t\in [0,1]\}$). It is easy to see that the state space $S_{x}$ is a face of $X_{1}^*$, thus $S_{x}\cap\text{ext}(X_{1}^*)=\text{ext}(S_{x})$ is non-empty by the Krein--Milman theorem, as $S_x$ is weak$^*$ compact. An element $x\in S(X)$ is said to be a smooth point in $X$ if $S_x$ is a singleton. The following characterizations exemplify the importance of studying smoothness in a Banach space. A point $x\in S(X)$ is smooth if and only if the norm on $X$ is G\^ateaux differentiable at $x$ (see \cite[Theorem~5.4.17]{M}). A Banach space $X$ is said to be a smooth Banach space if every $x\in S(X)$ is a smooth point in $X$. We say $X$ is rotund if $S(X)$ does not contain any line segment. The rotundity of a Banach space is a well-known geometric feature, and is also closely related to the concept of smoothness. If $X^*$ is rotund, then $X$ is smooth. A detailed study of this relationship can be seen in Chapter~5 of \cite{M}.

We will first concentrate on the scenario in which the state space has finitely many linearly independent elements. For $k\in \mathbb{N}$, $x\in S(X)$ is said to be a $k$-smooth point in $X$ if $\text{dim}(\text{span}\{S_x\})=k$ (see \cite{LR}). Clearly, the $1$-smoothness of a point coincides with the smoothness of that point. Let $k\in \mathbb{N}\setminus \{1\}$, a point $x\in S(X)$ is said to be a multismooth point of order $k$ if it is a $k$-smooth point. These notions have been extensively studied in \cite{DK, KS, LR}. 

More generally we study the situation when the state space is norm compact or weakly compact. For $x\in S(X)$, since $S_x$ is weak$^*$ compact, we will discuss the compactness in other two natural topologies. One of the motivations for studying this notion is the recent article \cite{R1}. In \cite{R1}, it is shown that in a $C^*$-algebra $A$ with unity, if the state space of unity is weakly compact, then $A$ is a finite-dimensional space.

We recall from \cite{FP} that the norm of $X$ is said to be strongly subdifferentiable (SSD) at $x\in S(X)$ if $\lim _{t\to 0^+}\frac{\|x+ty\|-1}{t}$ exists uniformly for $y\in X_1$. In a finite-dimensional space, it is easy to see that the norm is SSD at every point in $S(X)$. Another motivation to study this notion is \cite[Proposition~2.7]{FP}, where it is shown that if the norm on $X$ is SSD and all state spaces are norm compact, then strong subdifferentiability passes through the quotients ${X}/{Y}$ for a proximinal subspace $Y\subseteq X$. A Banach space $X$ is said to be Asplund if $X^*$ has the Radon-Nikod\'ym property (in short RNP). For a fundamental understanding of the RNP, we refer the reader to \cite{DU}. It is also shown in \cite{GGS} that if for all $x\in S({X})$, $S_{x}$ is weakly compact and the duality map $x\to S_x$ is upper semi-continuous, then $X$ is an Asplund space. We refer the reader to \cite{GGS} for any undefined notation. These justify an independent study of weak compactness or norm compactness of $S_{x}$.

We denote the canonical embedding of $X$ in to its bidual $X^{**}$ by identifying $X$ with its image under the map $J_{X}:X\to X^{**}$, where $J_{X}$ is the canonical embedding. For clarity, we sometimes write $J_{X}(X)$ to explicitly refer to the embedded copy of $X$ in $X^{**}$.
A Banach space $X$ is said to be Hahn-Banach smooth if every $x^*\in X^*$ has a unique norm-preserving extension from $X$ to $X^{**}$. In this case, $J_{X^*}(x^*)$ is the only norm-preserving extension of $x^*$ from $X$ to $X^{**}$, see \cite{SF} for details.

Let us introduce a few more notations that we will use throughout this article. Let $\{X_i\}_{i\in J}$ be a family of Banach spaces. We denote by $\bigoplus_{{c_0}{i\in J}}(X_{i})$, the space $\{(x_i):x_{i}\in X_{i} \text{ and } \|x_i\|\to 0 \}$, endowed with the norm $\|(x_i)\|=\sup_{i\in {J}}\|x_i\|$  for $(x_i)\in \bigoplus_{{c_0}i\in J}(X_{i})$. Let $\bigoplus_{{\ell^{1}}{i\in J}} (X_i)=\{(x_i): x_i\in X_i, \sum_{i\in J}\|x_i\|<\infty\}$ and let $\bigoplus_{{\ell^{\infty}{i\in J}}} (X_i)=\{(x_i): x_i\in X_i,\text{ sup}_{i\in J}\|x_{i}\|<\infty\}$, equipped with $\ell^{1} \text{ and } \ell^{\infty}$ norms, respectively. If $X_i=X$ for all $i\in J$, we use the notations $c_{0}(X)$, $\ell^1(X)$ and $\ell^{\infty}(X)$, respectively, for the above spaces. Let $(\Omega, \Sigma, \mu)$ be a measure space. We use $L^1(\Omega,X)$ to denote the space of all (equivalent classes of) $X$-valued Bochner integrable functions, equipped with the norm $\|f\|=\int_{\Omega}\|f(\omega)\|d(\mu(\omega))$ for $f\in L^1(\Omega,X)$. We refer the reader to \cite{DU} for more information on this topic. Let $\mu$ be a finite measure and suppose that $X^*$ has the RNP with respect to $\mu$. Then, by \cite[Theorem~IV.1.1]{DU}, $L^1(\Omega,X)^*=L^{\infty}(\Omega,X^*)$, the space of all (equivalence classes of) $X^*$-valued strongly measurable functions that are essentially bounded. It is known that if $X^*$ is separable, then $X^*$ has the RNP (see \cite[Theorem~III.3.1]{DU}). Hence we can apply the above identification when $X^*$ is separable. Whenever we use $L^1(\Omega,X)$, we consider $\mu$ to be non-atomic. For the atomic case, we discuss $\ell^1(X)$. Let $\Omega$ be a compact Hausdorff space. We use $C(\Omega,X)$ to denote the space of $X$-valued continuous functions, equipped with the supremum norm. 

In Section~2, we discuss smoothness and multismoothness in $\ell^1(X)$. We discuss the stability of these two notions in $\ell^1(X)$. In \cite[Theorem~1.1]{DK}, the authors provided a characterization of smooth points in $L^1([0,1],X)$ using Kuratowski and Ryll-Nardzewski's selection theorem. In Theorem~\ref{3t1}, we provide an alternative proof using Von Neumann's selection theorem. In Theorem~\ref{Th316}, we show that if $X^*$ is separable and $f\in L^{1}([0,1],X)$ with $\|f\|=1$, then $S_f$ is norm compact if and only if $f$ is a smooth point. Let $X$ be a Banach space such that $X=M\bigoplus_{\ell^{\infty}}N$ for some closed subspaces $M$ and $N$ of $X$. In Section~3, we start with the characterization of points whose state spaces are weakly compact in $X$. Let $Y$ be a closed subspace of $X$. In Theorem~\ref{2T2}, we see that under some geometric conditions on $Y$, the weakly compact state spaces of $Y$ continue to be weakly compact state spaces of $X$. We discuss the complete characterization of state spaces in $L^1(\mu)$, $C(\Omega,X)$, $\bigoplus_{{c_{0}}{i\in J}}(X_{i})$ and $\bigoplus_{{\ell^1}{i\in J}}(X_{i})$ with respect to the norm compactness and weak compactness.

\section{Smoothness and Multismoothness in spaces of vector-valued functions.}
In this section, we first discuss the smooth points of $L^{1}([0,1],X)$. Let $([0,1],\mathcal{A},\mu)$ be the Lebesgue measure space, where $\mathcal{A}$ is the Lebesgue $\sigma$-field and $\mu$ is the Lebesgue measure. Let $V$ be a Polish space (separable completely metrizable topological space). We denote by $\mathcal{B}_V$, the Borel $\sigma$-algebra and by $\mathcal{A}\otimes \mathcal{B}_V$, the product $\sigma$-algebra on $[0,1]\times V$ (see \cite[p. 83, 87]{S} for details). We recall two selection theorems from \cite{S} in our setup.

\begin{theorem}[Kuratowski and Ryll-Nardzewski]\cite[Theorem~5.2.1]{S}\label{KRN}
    Every $\mu$-measurable, closed valued (non-void) multifunction $F:[0,1]\longrightarrow 2^V$ admits a $\mu$-measurable selection such that $h(t)\in F(t)$ for all $t\in [0,1]$. 
\end{theorem}
\begin{theorem}[Von Neumann]\cite[Corollary~5.5.8]{S}\label{VNT}
   If $B\in \mathcal{A}\otimes \mathcal{B}_V$, then there exists a $\mu$-measurable selection $h:[0,1]\to V$ such that $(t,h(t))\in B$ for all $t\in [0,1]$.
\end{theorem}

 A characterization of the smooth points of $L^{1}([0,1],X)$ was given in \cite[Theorem~1.1]{DK}. In the proof, the authors made use of Kuratowski and Ryll-Nardzewski's measurable selection theorem. In the below, we provide an alternative proof to \cite[Theorem~1.1]{DK}, obtained as an application of Von Neumann's selection theorem.

Let $f\in L^1([0,1],X)$ and let $Z=\{t\in[0,1]:\|f(t)\|=0\}$. Then we say $f\neq0$ almost everywhere if $\mu(Z)=0$. Let $K$ be a Polish space and let $T:[0,1]\to 2^{K}$ be a set-valued map. We define the graph $G(T)=\{(t,k)\in [0,1]\times K: k\in T(t)\}$ of the map $T$. Given a complex number $z\in \mathbb{C}$, we use $Re(z)$ to denote its real part.

\begin{theorem}\label{3t1}
   Assume $X^*$ is separable. Let $ f \in L^{1}([0,1],X)$ be such that $ \|f\| = 1 $, and $f(t)\neq 0$ almost everywhere. Then $ f $ is smooth in $L^{1}([0,1],X)$ if and only if $ \frac{f(t)}{\|f(t)\|} $ is smooth in $X$ almost everywhere.
\end{theorem}
\begin{proof}
 Since the Lebesgue $\sigma$-field is complete, by discarding null sets if necessary, we consider all the functions are everywhere defined. Since $f\neq0$ almost everywhere, we will assume $f(t)\neq0$ for all $t\in [0,1]$. Let $A=\{t\in[0,1]:\frac{f(t)}{\|f(t)\|}\text{ is not smooth}\}$. Since $X$ is separable, by S. Mazur's result, we have that the set of all smooth points in $S(X)$ forms a $G_{\delta}$ set (see \cite{MS}). In particular, it is a measurable set. Now we consider a function $g:[0,1]\to \mathbb{C}$ such that $g(t)=\frac{f(t)}{\|f(t)\|}$ almost everywhere. Since $f:[0,1]\to X$ is a measurable function, we have that $g$ is a measurable function. Thus the preimage under $g$ of the set of all smooth points in $S(X)$ is measurable in $[0,1]$, that is, $A^{c}$ is measurable set, and hence $A$ is a measurable set. Fix $h\in S_{f}$. Then we have that $h(f)=\int_{[0,1]}h(t)(f(t))dt=1$. So we get
    \[
   1= \int_{[0,1]}\|f(t)\|dt=\int_{[0,1]}h(t)(f(t))dt\leq \int_{[0,1]}|h(t)(f(t))|dt\leq \int_{[0,1]}\|f(t)\|dt.
    \]
    This gives that $|h(t)(f(t))|=\|f(t)\|$ almost everywhere. Compairing the real parts of the equation $\int_{[0,1]}h(t)(f(t))dt=1$, we get that $\int_{[0,1]}Re(h(t)(f(t)))dt=1=\int_{[0,1]}\|f(t)\|dt$. This implies $\int_{[0,1]}(\|f(t)\|-Re(h(t)(f(t))))dt=0$. Since $\|f(t)\|-Re(h(t)(f(t)))\geq 0$, we have $Re(h(t)(f(t)))=\|f(t)\|$ almost everywhere. So we  have that $Re(h(t)(f(t)))=|h(t)(f(t))|=\|f(t)\|$ almost everywhere. Thus $h(t)(f(t))=\|f(t)\|$ almost everywhere. We may assume that $h(t)\in S_{\frac{f(t)}{\|f(t)\|}}$ for all $t\in [0,1]$. If $\mu(A)>0$, then we show that there exists another $h'\in S_{f}$ which differs from $h$ on $A$. As $X^*$ is a separable space, we have that $S(X^*)$ is a Polish space with respect to the norm topology. Let $ F:[0,1] \to 2^{S(X^*)} $ be defined as $F(t) = S_{\frac{f(t)}{\|f(t)\|}}$ for $t\in [0,1]$. The graph of $F$ is given by $G(F) = \{ (t, x^*)\in[0,1]\times S(X^*)  :x^*(f(t))= \|f(t)\| \}$.

We claim that $ G(F) $ is a measurable set. We will first show that $(t,x^*)\to x^*(f(t))$ is a measurable function by showing that it is the limit of a sequence of measurable functions. We first consider the case when $f = x_0 \chi_E$ for some measurable set $E$ and $x_{0}\in X$. Then $x^*(f(t))=x_0(x^*).\chi_E(t)$. Since we can consider $x_{0}\in X^{**}$ to be a continuous function on $S(X^*)$, we get $(t, x^*) \to x^*(f(t))$ is a measurable function. Since a simple function is a finite sum of functions of the above type and the stability of measurable functions under finite sums holds true (see \cite[Chapter~II]{DU}), we have that $(t,x^*) \to x^*(f(t))$ is measurable when $f$ is simple. In general, since $f$ is strongly measurable, we choose a sequence of simple functions $ \{s_n\}_{n\in \mathbb{N}} $ such that $s_n(t) \to f(t)$  almost everywhere in the norm. Thus $x^*(s_n(t)) \to x^*(f(t))$ for all $x^*\in S(X^*)$. It follows that $(t, x^*)\to x^*(f(t)) $ is measurable as it is the limit of a sequence of measurable functions (see \cite[Chapter~II]{DU}). 
Since $f$ is measurable and $\|.\|$ is continuous, we have $t \to \|f(t)\|$ composed with the projection, $(t,x^*)\to t$ is measurable. Therefore $ G(F) $ is precisely the zero set of the measurable function $(t,x^*) \to x^*(f(t)) - \|f(t)\| $, hence it is a measurable set.


Define $ F_1: [0,1] \to 2^{S(X^*)} $ by
\[
F_{1}(t) =
\begin{cases} 
S_{\frac{f(t)}{\|f(t)\|}}\setminus\{h(t)\}, & t \in A , \\
F(t), & t \notin A.
\end{cases}
\]
Then
\[
G(F_1) = G(F) \setminus \{(t, h(t)): t \in A\}.
\]
Since $ h $ is a measurable function and $G(F)$ is a measurable set, we get that $ G(F_1) $ is a measurable set. Applying Von Neumann's selection theorem to $G(F_1)$, we get a measurable selection $ h^{'}: [0,1] \to S(X^*)$ such that $h^{'}(t) \neq h(t)$  $\text{for all  } t \in A$. Since $X^*$ is separable, by Pettis measurability theorem (See \cite[Pg.~42]{DU}), we have that the function $h:[0,1]\to S(X^*)$ is strongly measurable. That is, $h\in L^{\infty}([0,1],X^*)$. This implies $ h^{'}$ and $h$ are two distinct elements in $S_{f}$ and hence $f$ is not smooth. 

Conversely, let $\mu(A)=0$ and let $g_{1},g_{2}\in S_{f}$. Then $\int_{[0,1]}{g_{i}(t)(f(t))dt=1=\int_{[0,1]}\|f(t)\|dt\text{ for }i=1,2}$. This implies $\int_{[0,1]}{(\|f(t)\|-g_{i}(t)(f(t)))dt=0}$ for $i=1,2$. Thus $g_{i}(t)(f(t))=\|f(t)\|$ almost everywhere. By the hypothesis $g_{1}(t)=g_{2}(t)$ almost everywhere because $\{t\in [0,1]: g_{1}(t)\neq g_{2}(t)\}\subseteq A$. Thus $f\in S(L^{1}([0,1],X))$ is smooth.
\end{proof}
    If $f\in S(L^1([0,1],X))$ is smooth, then $S_{f}$ is a singleton. Thus $S_{f}$ is norm compact. In the next theorem we show that the converse is also true. In particular, $S_{f}$ has finitely many independent vectors ($k$-smooth) implies $S_{f}$ is norm compact so it is a singleton. We begin by recalling that $(X_{1}^*, w^*)$ is metrizable whenever $X$ is a separable space (see \cite[Theorem~3.16]{RW}). Moreover, if $d^*$ denotes the corresponding metric, then we have $d^*(x^*, y^*)\leq \|x^*-y^*\|$ for all $x^*, y^*\in X_{1}^*$ (see \cite[Pg.~69]{RW}).
    
 We are grateful to the referee for the valuable suggestions incorporated in Lemma~\ref{L1} and Theorem~\ref{Th316}.
    
    \begin{lemma}\label{L1}
        Let $X$ be a separable space and let $f\in L^{1}([0,1],X)$ be a non-zero function such that $\|f\|=1$ and $f(t)\neq0$ almost everywhere. For each $k\in \mathbb{N}$, the set $A_{k}=\{t\in [0,1]:d^*(x_{t}^*, y_{t}^*)\geq \frac{1}{k}\textnormal{ for some }x_{t}^*, y_{t}^*\in S_{\frac{f(t)}{\|f(t)\|}}\}$ is a measurable set.
    \end{lemma}
    \begin{proof}
        Since $X$ is a separable space, we have that $(X_{1}^*, w^*)$ is metrizable. Let $d^*$ denotes the corresponding metric. For $k\in \mathbb{N}$, we define $B_{k}=\{x\in S(X):d^*(x_{1}^*, x_{2}^*)\geq \frac{1}{k}\textnormal{ for some }x_{1}^*, x_{2}^*\in S_{x}\}$. Observe that $B_{k}$ is closed. Indeed, let $\{x_{n}\}_{n\in \mathbb{N}}\subseteq B_{k}$ such that $x_{n}\to x$ in norm for some $x\in S(X)$. For each $n\in \mathbb{N}$, we have $x_{1, n}^*, x_{2, n}^*\in S_{x_{n}}$ such that $d^*(x_{1, n}^*, x_{2, n}^*)\geq \frac{1}{k}$. Since $(X_{1}^*, w^*)$ is weak$^*$ compact, we may assume that the sequences $\{x_{1, n}^*\}_{n\in \mathbb{N}}, \{x_{2, n}^*\}_{n\in \mathbb{N}}$ are both weak$^*$ convergent to $x_{1}^*, x_{2}^*$, respectively, for some $x_{1}^*, x_{2}^*\in X_{1}^*$. Since $d^*(x_{1, n}^*, x_{2, n}^*)\geq \frac{1}{k}$ for each $n\in \mathbb{N}$, we have $d^*(x_{1}^*, x_{2, }^*)\geq \frac{1}{k}$. Also, since $x_{1, n}^*\to x_{1}^*$ in the weak$^*$ topology and $x_{n}\to x$ in norm, we have $1=x_{1, n}^*(x_{n})\to x_{1}^*(x)$ which implies $x_{1}^*\in S_{x}$. Similarly, $x_{2}^*\in S_{x}$. So we get that $x\in B_{k}$. Thus $B_{k}$ is a closed set. Now we consider a function $g:[0,1]\to \mathbb{C}$ such that $g(t)=\frac{f(t)}{\|f(t)\|}$ almost everywhere. Since $f:[0,1]\to X$ is a measurable function, we have that $g$ is a measurable function. Consequently, $g^{-1}(B_{k})\subseteq [0,1]$ is a measurable set, that is, $A_{k}$ is a measurable set. This completes the proof.
    \end{proof}
    We will make use of Lemma~\ref{L1} and the fact that $d^*(x^*, y^*)\leq \|x^*-y^*\|$ for all $x^*, y^*\in X_{1}^*$ to establish the following theorem.
\begin{theorem}\label{Th316}
    Assume $X^*$ is separable. Let $f\in L^{1}([0,1],X)$ be such that $\|f\|=1$, and $f(t)\neq0$ almost everywhere. Then $S_f$ is norm compact if and only if $f$ is smooth.
\end{theorem}
\begin{proof}
 Since Lebesgue $\sigma$-field is complete, by discarding null sets if necessary, we consider all the functions are everywhere defined. Since $f\neq0$ almost everywhere, we will assume $f(t)\neq0$ for all $t\in [0,1]$. If $f$ is a smooth point, then $S_{f}$ is a singleton, hence a norm compact set. If $f$ is not smooth, then we show that $S_{f}$ is not norm compact by constructing a sequence in $S_{f}$ with no convergent subsequence. Fix $h\in S_{f}$, and let $A=\{t\in[0,1]:{\frac{f(t)}{\|f(t)\|}}\text{ is not smooth}\}$. From Theorem~\ref{3t1}, we have $\mu(A)>0$. For $k\in \mathbb{N}$, we define $A_{k}=\{t\in [0,1]:d^*(x_{t}^*, y_{t}^*)\geq \frac{1}{k}\textnormal{ for some }x_{t}^*, y_{t}^*\in S_{\frac{f(t)}{\|f(t)\|}}\}$. From Lemma~\ref{L1}, we have that each $A_{k}\subseteq [0,1]$ is measurable. It is easy to see that $A=\bigcup_{k\in \mathbb{N}}A_{k}$. Since $\mu(A)>0$, we have that $\mu(A_{k})>0$ for some $k\in \mathbb{N}$. Let $Z=A_{k}$, $2\epsilon_{0}=\frac{1}{k}$ and $B_{d^*}^{o}(h(t), \epsilon_{0})$ denotes the open $\epsilon_{0}$-ball centred at $h(t)$ in the weak$^*$ metric. Similarly, $B_{\|.\|}^{o}(h(t), \epsilon_{0})$ denotes the open $\epsilon_{0}$-ball centred at $h(t)$ in norm. For each $t\in Z$, the set $S_{\frac{f(t)}{\|f(t)\|}}\setminus B_{d^*}^{o}(h(t), \epsilon_{0})$ is non-empty and measurable. Since the metric $d^*$ satisfies $d^*(x^*,y^*)\leq \|x^*-y^*\|$ for all $x^*,y^*\in X_{1}^* $, we must have that the set $S_{\frac{f(t)}{\|f(t)\|}}\setminus B_{\|.\|}^{o}(h(t), \epsilon_{0})$ is non-empty and measurable for each $t\in Z$. Since $\mu$ is non-atomic and $\mu(Z)>0$, we can choose a sequence $\{Z_{n}\}_{n\in \mathbb{N}}$ of disjoint subsets of $Z$ such that $\mu(Z_{n})>0$ for each $n\in \mathbb{N}$. As $X^*$ is a separable space, we have that $S(X^*)$ is a Polish space with respect to the norm topology.

Define $ F_1: [0,1] \to 2^{S(X^*)} $ by
\[
F_{1}(t) =
\begin{cases} 
S_{\frac{f(t)}{\|f(t)\|}}\setminus B_{\|.\|}^{o}(h(t), \epsilon_{0}), & t \in Z_{1} , \\
S_{\frac{f(t)}{\|f(t)\|}}, & t \notin Z_{1}.
\end{cases}
\]
Then
\[
G(F_1) =\{(t,x^*):x^*(f(t))=\|f(t)\|\} \setminus \{(t,x^*):\|x^*-h(t)\|<\epsilon_{0}, t \in Z_1\}.
\]

We already have shown in the proof of Theorem~\ref{3t1} that the set $\{(t,x^*):x^*(f(t))=\|f(t)\|\}$ is a Borel set. Since the mappings $(t,x^*)\rightarrow x^*-h(t)$ and $\ (x^*-h(t))\rightarrow\|x^*-h(t)\|$ are measurable and continuous, respectively, we have $\{(t,x^*):\|x^*-h(t)\|<\epsilon_{0}, t \in Z_1\}$ is a measurable set. Hence $ G(F_1) $ is a measurable set. Applying Von Neumann's selection theorem to $G(F_{1})$ as in Theorem~\ref{3t1}, we get a selection $h^{'}_1: [0,1] \to S(X^*)$ such that $\|h^{'}_1(t)- h(t)\|\geq\epsilon_{0}\text{ for all  } t \in Z_1$. This implies $ \|h^{'}_1- h\|\geq \epsilon_{0}$. Define $ h_1: [0,1] \to X^* $ by
\[
h_{1}(t) =
\begin{cases} 
h^{'}_{1}(t), & t \in Z_{1} , \\
h(t), & t \notin Z_{1}.
\end{cases}
\]
Fix $h_{0}=h$. Then it is easy to see that $h_{0}$ and  $h_1$ are two distinct elements in $ S_f $ such that $\|h_1- h_{0}\|\geq\epsilon_{0}$ and $h_{0}(t)=h_{1}(t) \text{ for all } t\in Z_{1}^{c}$.
Continuing this process for each $Z_{n}$, we get a sequence $\{h_n\}_{n\geq0}$ such that $\|h_n-h_m\|\geq\epsilon_{0}$ for all $n\neq m$. Hence $S_{f}$ is not norm compact. This completes the proof.
\end{proof}
\begin{remark}\label{naddr}
In Section~\ref{Sec32}, we will see that if a function $f\in L^1([0,1],X)$ is not non-zero almost everywhere, then it can never be a smooth point. This discussion, along with the “not non-zero” case of Theorem~\ref{Th316} in a more general setting, will be considered in the same section.
\end{remark}

We now address the above discussion for vector-valued discrete situations. To begin, we state a well-known result that will help us in making our first observation. We give a short proof for the sake of completeness.

Given $Y\subseteq X$ a closed subspace and $y\in S(Y)$, we denote $S^{Y^*}_{y}=\{y^*\in S(Y^*):y^*(y)=1\}$.
\begin{lemma}\label{2l1}
    Let $Y$ be a proper subspace of $X$ such that $X=Y\bigoplus_{\ell^1} Z$ for some closed subspace $Z$ of $X$. Then, no point in $S(Y)$ is smooth in $X$.
\end{lemma}
\begin{proof}
Let $y\in S(Y)$. Then it is easy to check that $S_y=S^{Y^*}_y\times Z_1^*$. Hence the result follows.
\end{proof}

\begin{remark}\label{2p0}
 Let $J$ be an uncountable set. Consider $(x_i)\in {\bigoplus_{\ell^1}}_{i\in J}{X_i}$ and let $\Lambda=\{i\in J:\|x_i\|\neq 0\}$. Observe that $\Lambda$ is atmost countable. Now we can write ${\bigoplus_{\ell^1}}_{i\in J}X_i=\left({\bigoplus_{\ell^1}}_{i\in \Lambda}X_i\right) \bigoplus_{\ell^1} \left({\bigoplus_{\ell^1}}_{i\in \Lambda^c}X_i\right)$. We may consider $(x_i)\in {\bigoplus_{\ell^1}}_{i\in \Lambda}X_i$. Hence by Lemma~\ref{2l1}, $(x_{i})$ is not smooth in ${\bigoplus_{\ell^1}}_{i\in J}{X_i}$.
\end{remark}

Our next step is to examine the above study in the context of spaces of absolutely summable sequences of vectors. Let $\{S_i\}_{i\in \mathbb{N}}$ be a family of subsets of $X^*$. We define $\prod_{i\in \mathbb{N}}S_i=\{(x^*_i):x_i^*\in S_i\}$.
\begin{lemma}\label{nadd}
Let $(x_i)\in S(\ell^1(X_{}))$ such that $x_i\neq 0$ for all $i\in \mathbb{N}$. Then $S_{(x_i)}=\prod_{i=1}^{\infty}S_{\frac{x_i}{\|x_{i}\|}}$.
\end{lemma}
\begin{proof}
If $(x^*_{i})\in S_{(x_i)}$, then $\sum_{n=1}^{\infty}x_{i}^*(x_{i}) = \sum_{n=1}^{\infty}\|x_{i}\|x_{i}^*(\frac{x_{i}}{\|x_{i}\|})=1$. So we get
\[
1=\sum_{n=1}^{\infty}\|x_{i}\|x_{i}^*(\frac{x_{i}}{\|x_{i}\|})\leq \sum_{n=1}^{\infty}\|x_{i}\||x_{i}^*(\frac{x_{i}}{\|x_{i}\|})|\leq1 
\]
Since $\sum_{i=1}^{\infty}\|x_{i}\|=1$, we have that $|x_{i}^*(\frac{x_{i}}{\|x_{i}\|})|=1$ for each $i\in\mathbb{N}$. Also, by compairing real and imaginary part of the equation $\sum_{n=1}^{\infty}\|x_{i}\|x_{i}^*(\frac{x_{i}}{\|x_{i}\|})=1$ as in Theorem~\ref{3t1}, we get that $Re(x_{i}^*(\frac{x_{i}}{\|x_{i}\|}))=|x_{i}^*(\frac{x_{i}}{\|x_{i}\|})|=1$. This implies $x_{i}^*(\frac{x_{i}}{\|x_{i}\|})=1$. Consequently, $x_{i}^*\in S_{\frac{x_i}{\|x_{i}\|}}$. Thus we have $S_{(x_i)}\subseteq\prod_{i=1}^{\infty}S_{\frac{x_i}{\|x_{i}\|}}$. It is easy to show the other inclusion. Hence $S_{(x_i)}=\prod_{i=1}^{\infty}S_{\frac{x_i}{\|x_{i}\|}}$.
\end{proof}
\begin{proposition}\label{2p2}
    Let $(x_i)\in S(\ell^1(X_{}))$. Then $(x_i)$ is smooth in $\ell^1(X_{})$ if and only if $\|x_i\|\neq 0$ for all $i\in \mathbb{N}$ and $\frac{x_i}{\|x_i\|}$ is smooth in $X_{}^*$ for each $i\in \mathbb{N}$.
\end{proposition}
\begin{proof}
   Let $(x_{i})$ be smooth. Clearly, from Remark~\ref{2p0} $\|x_{i}\|\neq0$ for all $i\in \mathbb{N}$. The conclusion now follows from Lemma~\ref{nadd}.
   
   Conversely, let $\|x_{i}\|\neq0$ for all $i\in \mathbb{N}$ and let $\frac{x_i}{\|x_i\|}$ be smooth in $X_{}^*$ for each $i\in \mathbb{N}$. Since $S_{\frac{x_i}{\|x_{i}\|}}$ is a singleton for each $i\in \mathbb{N}$, by Lemma~\ref{nadd}, $S_{(x_{i})}$ is a singleton. Hence $(x_{i})$ is smooth.
\end{proof}
We now consider the above situation in the framework of multismoothness. First observe that any two distinct elements from a state space must be linearly independent. To see this, let $x^*,y^*\in S_{x}$ and $x^*=\alpha y^*$ for some $\alpha\in \mathbb{C}$. Since $x^*(x)=y^*(x)=1$, we have $\alpha=1$. This implies $x^*=y^*$. We will make use of this observation in the next theorem. 

The authors thank the referee for suggesting the second part of the following theorem. 
\begin{theorem}\label{th24}
    Let $(x_i)\in \ell^{1}(X_{})$ be such that $\|(x_i)\|=1$ and $\|x_{i}\|\neq 0$ for all $i\in\mathbb{N}$. Then $(x_i)$ is multismooth of finite order if and only if $\Lambda=\{i\in\mathbb{N}:\frac{x_{i}}{\|x_{i}\|}\text{ is not smooth}\}$ is finite and $\frac{x_{i}}{\|x_{i}\|}$ is multismooth of finite order for each $i\in \Lambda$. Moreover, if for each $i\in \Lambda$, order of $\frac{x_{i}}{\|x_{i}\|}$ is equal to $n_{i}$, then the order of the sequence $(x_{i})$ is equal to $\prod_{i\in \Lambda} n_{i}$.
\end{theorem}
\begin{proof} Let $(x_{i})$ be a multismooth point of finite order. Suppose if possible $\Lambda$ is an infinite set say, $\Lambda=\{i_{k}:k\in\mathbb{N}\}$. Fix some $(x^*_i)\in S_{(x_i)}$ and $j\in \mathbb{N}$. Define,
    $$
     M^{(j)}_{i}=
     \begin{cases}
        S_{\frac{x_{i}}{\|x_{i}\|}}, & i\in\{i_{1},i_{2},...i_{j}\},\\
        \{x^*_{i}\}, & \text{otherwise}.
     \end{cases}
     $$
     Let $E_{j}=\prod_{i=1}^{\infty} M^{(j)}_{i}$. Now we will show that $\text{span}(E_{j})\subsetneq \text{span}(E_{j+1})$ for all $j\in \mathbb{N}$. Let $z^*_{i_{j+1}}\in S_{\frac{x_{i_{j+1}}}{\|x_{i_{j+1}}\|}}$ be such that $z^*_{i_{j+1}}\neq x^*_{i_{j+1}}$. So we get $z^*_{i_{j+1}}\text{ and } x^*_{i_{j+1}}$ are two independent elements in $S_{\frac{x_{i_{j+1}}}{\|x_{i_{j+1}}\|}}$. Define $y^*_{n}=x^*_{n} \text{ for } n\neq i_{j+1} \text{ and } y^*_{i_{j+1}}=z^*_{i_{j+1}}$. Then it is easy to see that $(y^*_{n})_{n\geq1}\in \text{span}(E_{j+1})$ but $(y^*_{n})_{n\geq1}\notin \text{span}(E_{j})$. Hence $\text{span}(E_{j})\subsetneq\text{span}(E_{j+1})$, consequently $\text{dim}(\text{span}(E_{j+1}))>\text{dim}(\text{span}(E_{j}))$. Similarly, we get $\text{dim}(\text{span}(S_{(x_{i})}))>\text{dim}(\text{span}(E_{j}))$ for all $j\in \mathbb{N}$. From the last two inequalities, we have $\text{dim}(\text{span}(S_{(x_{i})}))>\text{dim}(\text{span}(E_{j}))\geq \text{dim}(\text{span}(E_{1}))+j-1$. Hence $\text{dim}(\text{span}(S_{(x_{i})}))$ is infinite as $j\to \infty$. Thus $(x_i)$ is multismooth of an infinite order, which is a contradiction to our assumption. Hence $\Lambda$ must be a finite set. By using the similar argument as above, we get, $\text{dim}(\text{span}(S_\frac{x_{k}}{\|x_{k}\|}))$ $\leq \text{dim}(\text{span}(S_{(x_{i})}))<\infty$ for each $k \in \mathbb{N}$. Thus $\frac{x_{k}}{\|x_{k}\|}$ is multismooth of finite order for each $k\in \Lambda$. 

Conversely, suppose $\Lambda$ is finite and $\frac{x_{i}}{\|x_{i}\|}$ is multismooth of finite order for each $i\in \Lambda$. By Lemma~\ref{nadd}, $
S_{(x_{i})} = \prod_{i=1}^{\infty} S_{\frac{x_i}{\|x_{i}\|}}$.
By our assumption, we have $S_{\frac{x_{i}}{\|x_{i}\|}}$ is a singleton for $i\in \Lambda^c$ and $\text{dim}(\text{span}(S_\frac{x_{i}}{\|x_{i}\|}))<\infty$ for $i\in \Lambda$. Since $\Lambda$ is finite, we have $\text{dim}(\text{span}(\prod_{i=1}^{\infty} S_{\frac{x_i}{\|x_{i}\|}})<\infty$. This implies $\text{dim}(\text{span}(S_{(x_{i})}))<\infty$. Consequently $(x_i)$ is multismooth of finite order.

Finally, let $\frac{x_{i}}{\|x_{i}\|}$ is multismooth of order $n_{i}$, for each $i\in \Lambda$. For $i\in \Lambda$, let $B_{i}$ denote the basis of span$(S_{\frac{x_{i}}{\|x_{i}\|}})$. Fix some $(x_{i}^*)\in S_{(x_{i})}$ and $j\in \mathbb{N}$ and define 
    \[
C_{j}=
     \begin{cases}
        B_{j}, & j\in \Lambda,\\
        \{x^*_{j}\}, & \text{otherwise}.
     \end{cases}
    \]
    We have that $B=\prod_{j=1}^{\infty}C_{j}\subseteq S_{(x_{i})}$ is a finite set such that card$(B)=\prod_{i\in \Lambda} n_{i}$. Now we show that $B$ is linearly independent. Let $\alpha_{1}(x_{i}^*(1))+ \alpha_{2}(x_{i}^*(2))+\alpha_{3}(x_{i}^*(3))+\dots+\alpha_{k}(x_{i}^*(k))=0$, where $\{\alpha_{i}\}_{1\leq i\leq k}\subseteq \mathbb{C}$ and $(x_{i}^*(n))\in B$ for $1\leq n\leq k$. For $i\in \Lambda$, by equating the $i^{th}$-coordinate in the above equation, we get $\alpha_{1}x_{i}^*(1)+ \alpha_{2}x_{i}^*(2)+\alpha_{3}x_{i}^*(3)+\dots+\alpha_{k}x_{i}^*(k)=0$. Since $x_{i}^*(n)\in B_{i}$ for $1\leq n\leq k$, we have that $\alpha_{i}=0$. This implies $B$ is linearly independent. A simple combinatorial argument shows that $B$ generates the space span$(S_{(x_{i})})$. This completes the proof.
\end{proof}
\section{On weak compactness of state spaces}
In this section, we first discuss the weak compactness of state spaces in a Banach space under certain geometric conditions on it. We also address the weak compactness of state spaces in vector-valued case, that is, in ${\bigoplus_{c_{0}}}_{i\in J}{X_i},\ {\bigoplus_{\ell^1}}_{i\in J}{X_i}$ and $C(\Omega,X)$.

\subsection{Preliminaries}
 Let $X=M\bigoplus_{\ell^{\infty}} N$, $\ell^{\infty}$-direct sum of the closed subspaces $M$ and $N$ of $X$. Then it is straightforward to check that $X^*=M^*\bigoplus_{\ell^1} N^*$, the $\ell^{1}$-direct sum of $M^*$ and $N^*$. For $E\subseteq X$, we denote the convex hull of $E$ by $\text{co}(E)$ and the set of all extreme points of $E$ by $\text{ext}(E)$. Notice that $S_{x}$ is a weak$^*$ compact subset of $X^*_{1}$, therefore weak compactness of $S_{x}$ is equivalent to saying that the weak and weak$^*$ topologies coincide on $S_{x}$. Similarly, norm compactness of $S_{x}$ is equivalent to saying that the norm and weak$^*$ topologies coincide on $S_{x}$. We apply this identification to the following result.

\begin{theorem}\label{2T1}
    Let $X$ be a Banach space such that $X=M\bigoplus_{\ell^{\infty}}N$. Let $x\in S(X)$ be such that $x=m+n$, where $m\in M$, $n\in N$. Then the following statements are true.
    \begin{enumerate}[label=(\alph*)]
    \item Let $\ ||n|| < 1$. Then $S_{x}$ is weakly compact in $X^*$ if and only if $S^{M^*}_{m}$ is weakly compact in $M^*$.
    \item Let $||n||=||m||= 1$. Then $S_{x}$ is weakly compact in $X^*$ if and only if $S^{M^*}_{m}$ is weakly compact in $M^*$ and $S^{N^*}_{n}$ is weakly compact in $N^*$.
    \end{enumerate}
\end{theorem}
\begin{proof} $(a)$ Let $\|n\|<1$. Then $\|m\|=1$. For $x^* \in S_{x}$, we have $x^* = m^* + n^*$ with $\|m^*\| + \|n^*\| = 1$ and  $x^*(x) = m^*(m) + n^*(n) = 1$, where $m^*\in M^*$, $n^*\in N^*$. Clearly, $m^*$ is non-zero. Now we claim that $n^*=0$. If not, then we can write $\|m^*\|\frac{ m^*(m)}{\|m^*\|} + \|n^*\| \frac{n^*(n)}{\|n^*\|} = 1$. Thus $\frac{m^*(m)}{\|m^*\|} = 1 \quad \text{and} \quad \frac{n^*(n)}{\|n^*\|} = 1$, which is a contradiction to $\|n\| < 1$. Thus we have $n^* = 0$. Hence $m^*\in S^{M^*}_m$. Consequently $S_{x} = S^{M^*}_{m}$. Thus $S_{x}$ is weakly compact in $X^*$ if and only if $S^{M^*}_{m}$ is weakly compact in $M^*$.  
    
$(b)$ Let $\|m\| = \|n\| = 1$. For $x^*=m^*+n^* \in S_{x}$, $m^*\in M^*$ and $n^*\in N^*$. Suppose $m^*=0$, then $x^*\in S_n^{N^*}$. Similarly, if $n^*=0$, then $x^*\in S_{m}^{M^*}$. Assuming $m^*\neq 0$ and $n^*\neq 0$, we write $x^* = \|m^*\| \left( \frac{m^*}{\|m^*\|} \right) + \|n^*\| \left( \frac{n^*}{\|n^*\|} \right)$. From the proof of part $(a)$, we have $\frac{m^*}{\|m^*\|} \in S^{M^*}_{m} \quad \text{and} \quad \frac{n^*}{\|n^*\|} \in S^{N^*}_{n}$. Thus $S_{x} \subseteq \operatorname{co} ( S^{M^*}_{m}  \cup  S^{N^*}_{n})$. The other inclusion is direct. Hence $S_{x} = \operatorname{co} ( S^{M^*}_{m}  \cup S^{N^*}_{n})$.
Now if $S_{x}$ is weakly compact, then we know that weak and weak$^*$ topologies coincide on $S_{x}$. Since $S^{M^*}_{m}$ and $S^{N^*}_{n}$ are weakly closed subsets of $S_{x}$, we have that $S^{M^*}_{m}$ and $S^{N^*}_{n}$ are weakly compact in $M^*$ and $N^*$, respectively. Conversely, suppose $S^{M^*}_{m}$ and $S^{N^*}_{n}$ are weakly compact. Then  by \cite[Theorem~3.20.$(a)$]{RW}, convex hull of finite union of weakly compact convex sets is weakly compact. So $S_{x}$ is weakly compact. This completes the proof.
\end{proof}

The subspace $M$ in Theorem~\ref{2T1} is known as an $M$-summand (see \cite[Definition~I.1.1]{HWW}). Let $Y$ be a closed subspace of a Banach space $X$. We denote by $Y^\#$, the set $\{x^*\in X^*:\|x^*\|=\|x^*|_Y\|\}$, where $\|x^*|_Y\|=\sup_{y\in S(Y)}|x^*(y)|$. We say that $Y$ has property-$(SU)$ if $Y^\#$ is linear. In this case, we have $X^*=Y^\#\oplus Y^\perp$ and if $x^*=y^{\#}+y^\perp$, $y^\perp\in Y^\perp$, $y^\#\in Y^{\#}$, then $\|x^*\|>\|y^{\#}\|$ whenever $y^\perp \neq 0$ (see \cite{O}). One can check that if $Y$ is a $M$-summand in $X$, then $Y^\#$ is linear. The linearity of $Y^\#$ implies that the canonical restriction map from $Y^{\#}$ to $Y^*$ is an isometry. For a recent work on property-$(SU)$ see \cite{DPR}. The following theorem is more general than part $(a)$ of Theorem~\ref{2T1}.
\begin{theorem}\label{2T2}
    Let $Y$ has the property-$(SU)$ in $X$ and $y\in S(Y)$. Then $S_y$ is weakly compact in $X^*$ if and only if $S^{Y^*}_y$ is weakly compact in $Y^*$.
\end{theorem}
\begin{proof}
  Suppose $Y$ has the property-$(SU)$ in $X$.  Then we have $X^*=Y^*\oplus Y^{\perp}$. Let $x^*=y^*+y^{\perp}\in S_y$, where $y^*\in Y^*$, $y^{\perp}\in Y^{\perp}$. Since $x^*\in S^{}_y$, we have $y^*(y)=1$, so $1\leq\|y^*\|$. Suppose if possible $y^{\perp}\neq 0$. Then we have $1\leq\|y^*\|<\|x^*\|$, which is a contradiction to $\|x^*\|=1$. So we have $y^{\perp}=0$ and $y^*\in S^{Y^*}_y$. Clearly, $S^{Y^*}_y=S_{y}$. Since $Y^*$ is a subspace of $X^*$, we have $S_{y}$ is weakly compact in $X^*$ if and only if $S^{Y^*}_{y}$ is weakly compact in $Y^*$.
\end{proof}
\begin{remark}
If a closed subspace $M\subseteq X$ is such that $M^{\#}$ is linear and $X^*=M^{\#}\bigoplus_{\ell^1}M^{\perp}$, then $M$ coincides with a class of Banach spaces called $M$-ideals studied in \cite{HWW}. Hence our result also applies for $M$-ideals.
\end{remark}

For the case of $L$-summand (see \cite{HWW} for definition),  we recall the following from \cite{R1}.

\begin{proposition}\label{2P4}\cite[Proposition~4]{R1}
    Let $X=Y\bigoplus_{\ell^1} Z$ and $y\in S(Y)$. Then $S_y$ is weakly compact in $X^*$ if and only if $S^{Y^*}_y$ is weakly compact in $Y^*$ and $Z$ is reflexive.
\end{proposition}
Arguments similar to the ones given in the proof of \cite[Proposition~4]{R1} also lead to the following proposition.
\begin{proposition}\label{2P5}
    Let $X=Y\bigoplus_{\ell^1} Z$ and $y\in S(Y)$. Then $S_y$ is norm  compact in $X^*$ if and only if $S^{Y^*}_y$ is norm compact in $Y^*$ and $Z$ is a finite-dimensional space.
\end{proposition}

\subsection{State Spaces of vector-valued functions}\label{Sec32}
 Let $(\Omega,\Sigma,\mu)$ be a measure space. Then for a Borel subset $A\subseteq\Omega$ with $\mu(A)>0$, we define $\Sigma_{A}=\{B\cap A:B\in \Sigma\}$ and $\mu^{'}=\mu|_{{\Sigma}_{A}}$. We use $L^{1}(A,,\mu^{'},X)$ to represent the space of Bochner integrable functions restricted to $A$. We denote by $L^{1}(A^{c},\mu^{''},X)$ the corresponding objects when $\mu(A^{c})>0$. For $f\in L^{1}(\mu,X)$, define $|f|:\Omega\to \mathbb{R}$ by $|f|(\omega)=\|f(\omega)\|$. It is easy to see that $|f|\in L^1(\mu)$.
 \begin{example}\cite[p.~170]{Ho}\label{52}
   Let $\mu$ be the Lebesgue measure on $[0,1]$ and $f\in S(L^1(\mu))$. Then $f$ is a smooth point in $L^1(\mu)$ if and only if $\mu(Z)=0$, where $Z=\{t\in [0,1]: f(t)=0\}$.
 \end{example}
\begin{proposition}\label{th39}
    Let $X$ be a Banach space and let $\mu$ be the Lebesgue measure on $[0,1]$. If $f\in S(L^1([0,1],X))$ is such that $S_f$ is either weakly compact or norm compact, then $|f|$ is smooth in $L^1(\mu)$.
\end{proposition}
\begin{proof}
    Let $Z=\{t\in [0,1]: f(t)=0 \}=\{t\in [0,1]: \|f(t)\|=0 \}$. Suppose $|f| $ is not smooth. From Example \ref{52}, we have $1>\mu(Z)>0$. Now we show that $S_{f}$ is neither norm compact nor weakly compact. We can write $L^1([0,1],\mu,X)=L^1(Z,\mu^{'},X)\bigoplus_{\ell^1} L^1(Z^{c},\mu^{''},X)$. Clearly, $f\in L^1(Z^{c},\mu^{''},X)$. Since $\mu$ is a non-atomic measure, we have that $L^1(Z,\mu^{'},X)$ is not a reflexive space. In particular, it is not finite-dimensional. Hence $S_{f}$ is neither norm compact nor weakly compact (see Proposition~\ref{2P4} and \ref{2P5}).  
\end{proof}

For a compact Hausdorff space $\Omega$, we recall that $C(\Omega,X)$ denote the space of all $X$-valued continuous functions on $\Omega$, equipped with the supremum norm. In the case, when $X=\mathbb{R \textnormal{ or }C}$, we simply write $C(\Omega)$. A fundamental result states that the dual can be described using Radon measures when \( X = \mathbb{R}\text{ or }\mathbb{C} \), i.e., $C(\Omega)^* \cong M(\Omega)$,
where \( M(\Omega) \) is the space of Radon measures on \( \Omega \). For general Banach spaces \( X \),  we have the identification $C(\Omega, X)^* \cong M(\Omega, X^*)$,
where \( M(\Omega, X^*) \) consists of \( X^* \)-valued Radon measures, equipped with the total variation  norm (see \cite[Theorem~1.7.1]{CP}).
A basic example of elements in \( C(\Omega, X)^* \) are elementary tensors. Given \( \mu \in M(\Omega) \) and \(x^* \in X^*\), an elementary tensor is given by
$
{(\mu\otimes x^*)}(f) = \int_\Omega x^*(f(\omega)) \, d\mu(\omega),
$
for all \( f \in C(\Omega, X) \). In particular, if $\omega\in\Omega$ and $x^*\in X^*$, then $(\delta_{\omega}\otimes x^*)(f)=x^*(f(\omega))$, where $\delta_{\omega}$ is the dirac measure concentrated at the point $\omega$. For $x^{**}\in X^{**}$ and a Borel set $A\subseteq \Omega$, let $(x^{**}\otimes\chi_{A})(\mu):=x^{**}(\mu(A))$. It is easy to check that $x^{**}\otimes\chi_{A}\in C(\Omega,X)^{**}$. In the next theorem, we characterize the weak compactness of state spaces in $C(\Omega, X)^*$ when $X$ satisfies $\overline{\textnormal{ext}}^{w^*}{(X_{1}^*)}\subseteq S(X^*)$. Throughout, we use $\Gamma$ to denote the unit sphere in the complex plane. We recall from \cite[Theorem~8.4]{C}, that ext$(C(\Omega)_{1}^*)=\{t\delta_{\omega}:t\in \Gamma \textnormal{ and }\omega\in \Omega\}$. Before presenting the theorem, consider the following proposition.
\begin{proposition}\label{P38}
   Let $A\subseteq C(\Omega)$ be a closed subspace such that $1\in A$. Then every extreme point of $A_{1}^*$ is the restriction to $A$ of an extreme point of $C(\Omega)_{1}^*$. Moreover, we have $\overline{\textnormal{ext}}^{w^*}{(A_{1}^*)}\subseteq S(A^*)$.
\end{proposition}
\begin{proof}
 Let $a^*\in A_{1}^*$ is an extreme point. Let $K=\{\mu\in C(\Omega)_{1}^*:\mu|_{A}=a^*\}$. By Hahn-Banach theorem, we have that $K\subseteq C(\Omega)_{1}^*$ is a non-empty weak$^*$ compact convex set. We first show that $K$ is a face of $C(\Omega)_{1}^*$. Let $\mu\in K$ such that $\mu=t\mu_{1}+(1-t)\mu_{2}$ for some $\mu_{1}, \mu_{2}\in C(\Omega)_{1}^*$. Restricting both the sides to $A$, we get that $a^*=t{\mu_{1}}|_{A}+(1-t){\mu_{2}}|_{A}$. Since $a^*\in A_{1}^*$ is an extreme point, we have ${\mu_{1}}|_{A}={\mu_{2}}|_{A}=a^*$. Thus $\mu_{1}, \mu_{2}\in K$. So we get that $K$ is a face of $C(\Omega)_{1}^*$. Since $K$ is a weak$^*$ compact convex subset, so it contains an extreme point of the dual unit ball $C(\Omega)_{1}^*$. Hence we get that $a^*$ is the restriction to $A$ of an extreme point of $C(\Omega)_{1}^*$. Next we observe that $\overline{\textnormal{ext}}^{w^*}(C(\Omega)_{1}^*)\subseteq S(C(\Omega)^*)$. Indeed, let $\{t_{\alpha}\delta_{\omega_\alpha}\}_{\alpha\in J}$ be a net. Since $\Gamma$ and $\Omega$ are both compact, we may assume that $\{t_{\alpha}\}$ and $\{\omega_{\alpha}\}$ are both converging to $t$ and $\omega$, respectively, for some $t\in \Gamma$ and $\omega\in \Omega$. So we get that $t_{\alpha}\delta_{\omega_\alpha}\to t\delta_{\omega}$ in the weak$^*$ topology. This implies that weak$^*$ accumulation points of the set $\textnormal{ext}(C(\Omega)_{1}^{*})$ are of norm one. To see that $\overline{\textnormal{ext}}^{w^*}{(A_{1}^*)}\subseteq S(A^*)$, let $\{a^*_{\alpha}\}_{\alpha\in J}\in \textnormal{ext}(A_{1}^*)$ be a net such that $a_{\alpha}^*\to a_{0}^*$ in the weak$^*$ topology for some $a_{0}^*\in A_{1}^*$. Since $1\in A$, we have $a_{\alpha}^*(1)\to a_{0}^*(1)$. From the above we have that $a_{\alpha}^*=t_{\alpha}\delta_{\alpha}$ for $\{t_{\alpha}\}\subseteq \Gamma$ and $\{\omega_{\alpha}\}\subseteq \Omega$. This implies 
 $
 |a_{\alpha}^*(1)|=|t_{\alpha}|=1.
 $
 So we get that $|a_{0}^*(1)|=1$. Consequently, $\|a_{0}^*\|=1$. This completes the proof.
\end{proof}
The following remark shows that the converse of Proposition~\ref{P38} does not hold in general. For a convergent scalar valued sequence $(x_{i})$, we use the symbol `$x_{\infty}$' to denote its limit.
\begin{remark}
    Let $c$ be the space of all complex valued convergent sequences, equipped with the supremum norm and let $A=\{(x_{i})\in c:x_{\infty}=\frac{x_{1}+x_{2}}{2}\}$. Clearly, $A\subseteq c$ is a closed subspace containing every constant sequence. Now we define $e_{\infty}:c\to \mathbb{C}$ by $e_{\infty}((x_{i}))=x_{\infty}$. It easy to see that $\|e_{\infty}\|=1$ and it is an extreme point of $c_{1}^*$. Since $e_{\infty}|_{A}=\frac{e_{1}+e_{2}}{2}$, where $e_{1}, e_{2}$ are the evaluation functionals at $i=1$ and $i=2$, respectively, so we have $e_{\infty}|_{A}\in A_{1}^*$ is not an extreme point.
\end{remark}
 \begin{theorem}
 Let $f \in S(C(\Omega,X))$ and $ \Lambda = \{ \omega \in \Omega : \|f(\omega)\| = 1 \} $. If $ \Lambda $ is finite and $S_{f(w)}$ is weakly compact for each $w\in \Lambda$, then $ S_f $ is weakly compact. In addition, if we assume $\overline{\textnormal{ext}}^{w^*}{(X_{1}^*)}\subseteq S(X^*)$, then the converse is also true.
 \end{theorem}
 \begin{proof} 
Suppose $ \Lambda $ is finite and $S_{f(\omega)}$ is weakly compact for each $\omega\in \Lambda$. If $ \mu \in \text{ext}(S_f)$, then $ \mu = \delta_{\omega}\otimes x^* $ for some $ \omega \in \Omega $ and $ x^* \in \text{ext}(X_{1}^*)$ (see \cite[Theorem~1.1.$(a)$]{RS}). Now $\int_{\Omega} f \ d(\delta_{\omega}\otimes x^*) = 1$ that is, $x^* ( f(\omega)) = 1$. This implies $\|f(\omega)\| = 1$. Hence $ \omega \in \Lambda $ and $x^*\in \text{ext}(S_{f(\omega)})$, so we get $\text{ext} ( S_f) \subseteq \bigcup_{\omega\in\Lambda}\{\delta_{\omega}\otimes \text{ext}(S_{f(\omega)}) \}$. Since the convex hull of a finite union of weakly compact convex sets is weakly compact (see \cite[Theorem~3.20.$(a)$]{RW}), we have that ${\text{co}^{}}(\bigcup_{\omega\in\Lambda}^{}\delta_{\omega}\otimes S_{f(\omega)})$ is weakly compact. Since ${\text{co}^{}}(\bigcup_{\omega\in\Lambda}^{}\delta_{\omega}\otimes S_{f(\omega)})$ is weak$^*$ closed as $\Lambda$ is finite, one can easily check that $S_{f} =\overline{\text{co}}^{w^*}(\text{ext}(S_{f}))\subseteq {\text{co}^{}}(\bigcup_{\omega\in\Lambda}^{}\delta_{\omega}\otimes S_{f(\omega)})\subseteq S_{f}$. So we get $S_{f}={\text{co}^{}}(\bigcup_{\omega\in\Lambda}^{}\delta_{\omega}\otimes S_{f(\omega)})$. Hence $S_{f}$ is weakly compact. 
    
    Now assume $\overline{\textnormal{ext}}^{w^*}{(X_{1}^*)}\subseteq S(X^*)$ and suppose that $ S_f $ is weakly compact. We argue by contradiction to show that $ \Lambda $ is finite. Suppose $ \Lambda $ is infinite. We first show that each $ \omega \in \Lambda $ is an isolated point. For $ \{\omega_\alpha\}_{\alpha\in J}\subseteq\Lambda $, if $ \omega_\alpha \to \omega $ and $ \omega_\alpha $'s are all distinct, then $\delta_{\omega_\alpha} \to \delta_{\omega}$ in the weak$^*$ topology. Choose $x_{\alpha}^*\in \text{ext}(S_{f(\omega_{\alpha})})$. Then $\delta_{\omega_{\alpha}}\otimes x_{\alpha}^*\in S_{f}$ for each $\alpha\in J$. Since $X_{1}^*$  is w$^*$-compact and $\overline{\text{ext}}^{\omega^*}{(X_{1}^*)}\subseteq S(X^*)$, there exists a subnet of $\{x_{\alpha}^*\}_{\alpha \in J}$, again denoted by $\{x_{\alpha}^*\}_{\alpha \in J}$ and $x^*\in S(X^*)$ such that $x_{\alpha}^*\to x^*$ in the weak$^*$ topology. Clearly, $x_{\alpha}^*(g(\omega_{\alpha}))\to x^*(g(\omega))$ for all $g\in C(\Omega,X)$ as $g(\omega_\alpha)\to g(\omega)$ in norm. So we get that $\delta_{\omega_{\alpha}}\otimes x_{\alpha}^*\to \delta_{\omega}\otimes x^*$ in the weak$^*$ topology. Since $S_f$ is weakly compact, we have $\delta_{\omega_{\alpha}}\otimes x_{\alpha}^*\to \delta{_\omega}\otimes x^*$ in the weak topology. Let $x^{**}\in S(X^{**})$ be such that $x^{**}(x^*)=1$ and let $x^{**}\otimes\chi_{\{\omega\}} \in C(\Omega,X)^{**}$. Then we get $x^{**}\otimes\chi_{\{\omega\}}\left( \delta_{\omega_{\alpha}}\otimes x_{\alpha}^*\right) \to x^{**}\otimes\chi_{\{\omega\}}\left( \delta_{\omega}\otimes x^*\right)$. It follows that $x^{**}(x^*)=0$, which is a contradiction. Thus $\Lambda$ is a compact discrete subset of $\Omega$, hence $ \Lambda $ must be finite. Next we show that $S_{f(\omega)}$ is weakly compact for each $\omega\in \Lambda$. Fix $\omega\in \Lambda$ and let $\{x_{\alpha}^*\}_{\alpha\in J}\subseteq S_{f(\omega)}$ be a net and $x^*\in S_{f(\omega)}$ be such that $x_{\alpha}^*\to x^*$ in the weak$^*$ topology. Then we have $\{\delta_{\omega}\otimes x_{\alpha}^*\}_{\alpha\in J}\subseteq S_{f}$ and $\delta_{\omega}\otimes x_{\alpha}^*\to\delta_{\omega}\otimes x^*$ in the weak$^*$ topology as $x_{\alpha}^*(f(\omega))\to x^*(f(\omega))$ for all $f\in C(\Omega,X)$. Since $S_{f}$ is weakly compact, we have $\delta_{\omega}\otimes x_{\alpha}^*\to\delta_{\omega}\otimes x^*$ in the weak topology. For $\phi\in X^{**}$, $\phi\otimes\chi_{\{\omega\}}\in C(\Omega,X)^{**}$, so $\phi\otimes\chi_{\{\omega\}}\left( x_{\alpha}^*\otimes\delta_{\omega}\right) \to \phi\otimes\chi_{\{\omega\}}\left( \delta_{\omega}\otimes x^*\right)$, which implies $\phi(x_{\alpha}^*)\to \phi(x^*)$. Hence $x_{\alpha}^*\to x^*$ in the weak topology. This completes the proof.
 \end{proof}
 By using similar arguments and the fact that norm compactness implies weak compactness, we have the following corollary.
\begin{corollary}
    Let $f \in S(C(\Omega,X))$ and $ \Lambda = \{ \omega \in \Omega : \|f(\omega)\| = 1 \} $. If $ \Lambda $ is finite and $S_{f(w)}$ is norm compact for each $w\in \Lambda$, then $ S_f $ is norm compact. In addition, if we assume $\overline{\textnormal{ext}}^{w^*}{(X_{1}^*)}\subseteq S(X^*)$, then the converse is also true.
\end{corollary}
\subsection{More on direct sums} Now we will see the weak and norm compactness of state spaces in vector-valued discrete sums. Let $X=\prod_{i\in J}X_{i}$. For $j\in J$ and $x^*\in X_{j}^*$, we denote by $e_{j}x^*$ the element of $\prod_{i\in J}X_{i}^*$ whose $i^{th}$ coordinate is $x^*$ if $i=j$ and $0$ otherwise.
\begin{theorem}
Let $(x_i) \in \bigoplus_{{c_0}{i\in J}}X_{i}$ with $\| (x_i) \| = 1$ and $\Lambda = \{ i \in J : \| x_i \| = 1 \}$. Then $ S_{(x_i)} $ is weakly compact if and only if $ S_{x_i} $ is weakly compact for each $ i \in \Lambda $.
\end{theorem}
\begin{proof}
It is apparent that $\Lambda = \{ i \in \mathbb{N} : \| x_i \| = 1 \}$ is finite as $\|(x_i)\|\rightarrow0$. Applying similar arguments to the ones used in the proof of Theorem~\ref{2T1}.$(b)$, we get that $$S_{(x_i)} = \operatorname{co} \left( \bigcup_{i \in \Lambda} e_iS_{x_i} \right),$$ where $e_iS_{x_i}=\{e_{i}x^*:x^*\in S_{x_{i}}\}$.
Let $S_{(x_i)}$ be weakly compact and fix $i\in \Lambda$. Let $\{x_{\alpha}^*\}_{\alpha\in J}\subseteq S_{x_{i}}$ be a net and $x^*\in S_{x_{i}}$ be such that $x_{\alpha}^*\to x^*$ in the weak$^*$ topology. Consider $\{e_{i}x_{\alpha}^*\}_{\alpha\in J}\subseteq S_{(x_{i})}$. Then we have $e_{i}x_{\alpha}^*\to e_{i}x^*$ in the weak$^*$ topology. Since $S_{(x_i)}$ is weakly compact, $e_{i}x_{\alpha}^*\to e_{i}x^*$ in the weak topology. For any $x^{**}\in X_{i}^{**}$, we take $(e_{i}x^{**})\in (\bigoplus_{c_{0}{i\in J}}(X_{i}))^{**}$. Then we have $e_{i}x^{**}(e_{i}x_{\alpha}^*)\to e_{i}x^{**}(e_{i}x^*)$. Consequently $x^{**}(x_{\alpha}^*)\to x^{**}(x^*)$ and hence $x_{\alpha}^*\to x^*$ in the weak topology. This completes the proof.

Conversely, let $ S_{x_i} $ be weakly compact for each $ i \in \Lambda $. Since $ S_{(x_i)} $ is convex hull of a finite union of weakly compact convex sets, we have that $S_{(x_i)}$ is weakly compact by \cite[Theorem~3.20.$(a)$]{RW}.
\end{proof}
As another application of \cite[Theorem~3.20.$(a)$]{RW}, we have the following.
\begin{corollary}
    Let $(x_i) \in \bigoplus_{c_{0}{i\in J}}(X_{i})$ be such that $\| (x_i) \| = 1$ and let $\Lambda = \{ i \in \mathbb{N} : \| x_i \| = 1 \}$. Then $ S_{(x_i)} $ is norm compact if and only if $ S_{x_i} $ is norm compact for each $i\in\Lambda$.
\end{corollary}
 In the next theorem, we will derive the necessary and sufficient condition for the weak compactness of the state spaces in $\ell^1(X)$.
\begin{theorem}\label{3t12}
    Let $ (x_i) \in {\ell^{1}}(X)$ be such that $ \| (x_i) \| = 1 $, and let $ \|x_{i}\|\neq 0$ for all $i\in{\mathbb{N}}$. Then $ S_{(x_i)}$ is weakly compact in $ {{\ell^{\infty}}}(X_{}^*) $ if and only if $ S_{\frac{x_i}{\|x_i\|}} $ is weakly compact in $X^*$ for each $i\in{\mathbb{N}}$ and $\textup{diam}\left(S_{\frac{x_i}{\|x_i\|}}\right) \to 0 $.
\end{theorem}
\begin{proof}
From Proposition~\ref{2p2}, $S_{(x_i)} = \prod_{i=1}^{\infty} S_{\frac{x_i}{\|x_{i}\|}}$. Let $S_{(x_i)}$ be weakly compact and $\Pi_{j}:\prod_{i=1}^{\infty}S_{\frac{x_{i}}{\|x_{i}\|}}\rightarrow S_{\frac{x_{j}}{\|x_{j}\|}}$ be the canonical projection. Since $S_{(x_i)}$ is weakly compact, we get that $S_{\frac{x_i}{\|x_i\|}}$ is weakly compact for each $i\in\mathbb{N}$. Suppose that $\text{diam}\left(S_\frac{x_{i}}{\|x_{i}\|}\right)\not\to0$. There exists $\epsilon_{0}>0$ such that $\text{diam}\left(S_\frac{x_{i}}{\|x_{i}\|}\right)\geq\epsilon_0$ for infinitely many $i\in\mathbb{N}$. By re-enumerating the subscripts, we can  assume that $\text{diam}\left(S_\frac{x_{i}}{\|x_{i}\|}\right)\geq\epsilon_0$ for all $i\in\mathbb{N}$. Let $\{x_{i}^*,y_{i}^*\}\subseteq S_\frac{x_{i}}{\|x_{i}\|}$ be such that $x_{i}^*\neq y_{i}^* \text{ and } \|x_{i}^*-y_{i}^*\|\geq\frac{\epsilon_0}{2}$. Fix $n\in \mathbb{N}$ and define,\\
\[
x^*_{i}(n) =
\begin{cases} 
    x^*_{i}, & \text{if } 1\leq i \leq n, \\
    y^*_{i}, & \text{if } i > n.
\end{cases}
\]
 We write $(x^*(n))=(x^*_{i}(n))_{i\geq1}$. Thus we get a sequence $\{(x^*(n))\}_{n\geq1}\subseteq S_{(x_{i})}$ such that $(x^*_{i}(n))\rightarrow (x^*_{i})$ in the weak$^*$ topology as $n\to \infty$. Let $\Delta_{k}$ be an arbitrary convex combination of the elements of $\{(x^*(n))\}_{n\geq1}$. Since for each $n\in\mathbb{N}$, $x^*_{i}(n)=y_{i}^*$ for $i>n$, we have $\|\Delta_{k}-(x^*_{i})\|_{\infty}\geq\frac{\epsilon_0}{2}$.
Hence $\{(x^*(n))\}_{n\geq1}$ has no weakly convergent subsequence, as there is no sequence of convex combinations of the elements of $\{(x^*(n))\}_{n\geq1}$ converging to $(x^*_{i})$ in the norm. Thus by Mazur's Theorem (see \cite[Theorem~3.13]{RW}), $S_{(x_{i})}$ is not weakly compact.

To prove the converse, let $\{ (x^*(\alpha)) \}_{\alpha\in J}$ be a net in $S_{(x_i)}$. For $\alpha\in J$, $x_{i}^*(\alpha)$ denotes the $i^{th}$-coordinate of $(x^*(\alpha))$. Since $S_{(x_i)}$ is weak$^*$ compact, we may assume $\{ (x^*(\alpha) )\}_{\alpha\in J}$ is weak$^*$ convergent. Let $(x^*)\in S_{(x_i)}$ be such that $(x^*_{}(\alpha))\rightarrow (x^*_{})$ in the weak$^*$ topology. We show that $(x^*_{}(\alpha))\rightarrow (x^*_{})$ in the weak topology. It is known that $(\ell^\infty(X^*))^* = \ell^{1}(X^{**})\bigoplus_{\ell^{1}}c_{0}(X^*)^{\perp}$. Let $\Lambda\in(\ell^\infty(X^*))^*$ be such that $\Lambda = \Lambda' \bigoplus_{\ell^1} \Lambda''$, where $\Lambda'\in \ell^{1}(X^{**})$, $\Lambda''\in c_{0}(X^*)^{\perp}$. It is enough to show that for each $ \epsilon > 0 $,  $ |\Lambda((x^*(\alpha)) - (x^*))| < \epsilon $ for all but finitely many $ \alpha $. Choose $ K \in \mathbb{N} $ such that $\text{diam}\left(S_{\frac{x_i}{\|x_{i}\|}}\right) < \frac{\epsilon}{2C}$ for all $ i > K $, where $ C = \| \Lambda' \|_{1} $. We observe that for each $\alpha\in J, \{(x_{i}^*(\alpha)-x_{i}^*)\}_{i\geq1}\in c_{0}(X^*)$ because $\text{diam}\left(S_{\frac{x_i}{\|x_{i}\|}}\right)\to0$. Thus
\begin{align*}
    |\Lambda((x^*(\alpha)) - (x^*))|&=|\Lambda^{'}((x^*(\alpha)) - (x^*))|\\
&\leq \sum_{i=1}^{K-1} |\Lambda^{'}_i  ( x_i^*(\alpha) - x_i^*)|+ \sum_{i=K}^{\infty} \| \Lambda^{'}_i \| \cdot \| x_i^*(\alpha) - x_i^* \|\\
&\leq \sum_{i=1}^{K-1} | \Lambda^{'}_i ( x_i^*(\alpha) - x_i^*)|+\frac{\epsilon}{2}.
\end{align*}
 Since $S_{\frac{x_i}{\|x_{i}\|}} $ is weakly compact for each $1\leq i\leq K-1$, then $|\Lambda'_{i}(x_{i}^*(\alpha) -x_{i}^*)|<\frac{\epsilon}{2(K-1)}$ for all but finitely many $\alpha\in J$. Consequently, $|\Lambda(x^*(\alpha)) - (x^*)| < \epsilon$ for all but finitely many $\alpha\in J$. Hence $S_{(x_i)}$ is weakly compact.
\end{proof}
\begin{corollary}
   Let $X$ be a reflexive space. Let $(x_i) \in \ell^{1}(X_{}) $ be such that $ \| (x_i) \| = 1 $. Then $ S_{(x_i)}$ is weakly compact in $ \ell^\infty(X^*) $ if and only if  $\Lambda=\{i\in\mathbb{N:}\|x_{i}\|=0\}$ is finite and $ \textup{diam}\left(S_{\frac{x_i}{\|x_i\|}}\right) \to 0 $.
\end{corollary}
\begin{proof}
    Let $\Lambda=\{i\in\mathbb{N:}\|x_{i}\|=0\}$. Then we can write $\ell^1(X)=\ell^1(\Lambda^c,X)\bigoplus_{\ell^1}\ell^1(\Lambda,X)$. The result now follows from Proposition~\ref{2P4} and Theorem~\ref{3t12}.
\end{proof}
 An argument similar to the one given during the proof of Theorem~\ref{3t12} leads us to the following observation for norm compactness.
\begin{proposition}\label{3c4}
Let $ (x_i) \in {{\ell^{1}}}(X_{})$ with $ \| (x_i) \| = 1 $, and $ \|x_{i}\|\neq 0$ for all $i\in{\mathbb{N}}$. Then $S_{(x_i)}$ is norm compact in $ {{\ell^{\infty}}}(X_{}^*) $ if and only if $S_{\frac{x_i}{\|x_i\|}} $ is norm compact in $X_{}^*$ for each $i\in{\mathbb{N}}$ and $ \textup{diam}\left(S_{\frac{x_i}{\|x_i\|}}\right) \to 0 $.
\end{proposition}

\begin{proof}
 The only if part follows from a similar argument to the one  used in Theorem~\ref{3t12}.

For the converse, let us assume that $S_{\frac{x_i}{\|x_i\|}} $ is norm compact in $X_{}^*$ for each $i\in{\mathbb{N}}$ and  $\textup{diam}\left(S_{\frac{x_i}{\|x_i\|}}\right) \to 0$. To see that  $S_{(x_i)}$ is norm compact, let $\{(x^*(n))\}_{n\in\mathbb{N}} \subseteq S_{(x_i)}$ be a sequence. For each $k\in \mathbb{N}$, since $S_{\frac{x_k}{\|x_k\|}}$ is norm compact, we may assume that $\{x^*_{k}(n)\}_{n\in \mathbb{N}}$ is convergent to $x_{k}^*$ in norm for some $x_{k}^*\in S_{\frac{x_{k}}{\|x_{k}\|}}$. Now we claim that the sequence $\{(x^*(n))\}_{n\in\mathbb{N}}$ is converging to $\{x_{k}^*\}_{k\in\mathbb{N}}$ in norm. Fix $\epsilon>0$. Since $\textup{diam}\left(S_{\frac{x_i}{\|x_i\|}}\right) \to 0$, there exists  $m_{1}\in \mathbb{N}$ such that $\|x_{k}^*(n)-x_{k}^*\|<\frac{\epsilon}{2}$ for each $n\in \mathbb{N}$ and $k\geq m_{1}$. Since $x_{k}^*(n)\to x_{k}^*$ as $n\to \infty$ in norm for each $1\leq k\leq m_{1}$, there exists $m_{2}\in\mathbb{N}$ such that $\|x_{k}^*(n)-x_{k}^*\|<\frac{\epsilon}{2}$ for each $n\geq m_{2}$ and $1\leq k\leq m_{1}$. For each $n\geq m_{2}$, we have $\|x_{k}^*(n)-x_{k}^*\|\leq \frac{\epsilon}{2}$ for all $k\in \mathbb{N}$. Consequently, $\|\{x^*(n)\}-\{x_{k}^*\}\|_{\infty}\leq \frac{\epsilon}{2}<\epsilon$ for all $n\geq m_{2}$. Thus $S_{(x_{i})}$ is norm compact. This completes the proof.
\end{proof}

The following is immediate from Proposition~\ref{2P5} and Proposition~\ref{3c4}.
\begin{corollary}
    Let $X$ be a finite-dimensional space. Let $(x_i) \in \ell^{1}(X) $ be such that $ \| (x_i) \| = 1 $. Then $ S_{(x_i)}$ is norm compact in $ \ell^\infty(X^*) $ if and only if  $\Lambda=\{i\in\mathbb{N}:\|x_{i}\|=0\}$ is finite and $ \textup{diam}\left(S_{\frac{x_i}{\|x_i\|}}\right) \to 0 $.
\end{corollary}
\begin{problem}
In Theorem~\ref{3t1}, we have characterized smooth point in $L^{1}([0,1],X)$. In Theorem~\ref{Th316}, we characterized the norm compactness of $S_f$ for $f\in S(L^{1}([0,1],X))$. But a result similar to Theorem~\ref{Th316} for weak compactness is not known.
\end{problem}
\begin{problem}
    The conclusions of Theorem~\ref{3t1} and Theorem~\ref{Th316} are not known when $X^*$ is not separable.
\end{problem}
\vspace{2mm}
\noindent\small{\bf Conflict of interest:} The authors declare that there is no conflict of interest.

\subsection*{Acknowledgements}
The authors thank the referee for their thorough review. Their insightful comments have significantly improved the clarity and precision of several arguments in the paper. The authors also thank Prof. T. S. S. R. K. Rao and Dr. Priyanka Grover for many helpful discussions and valuable suggestions. The work was done when the first author was a Research Associate in the project Classification of Banach spaces using differentiability, funded by the Anusandhan National Research Foundation (ANRF), CRG/2023/000595, India (PI: T. S. S. R. K. Rao, Co PI: Priyanka Grover).


\vspace{4mm}
\textbf{Current address of Dr. Daptari:} \tiny{KATSUSHIKA DIVISION, INSTITUTE OF ARTS AND SCIENCES, TOKYO UNIVERSITY OF SCIENCE, TOKYO 125-8585, JAPAN}  

\end{document}